\documentclass[10pt]{amsart}            

\usepackage{amssymb,amsmath,amsthm}
\usepackage{hyperref}
\usepackage{color}
\usepackage{mathrsfs}

\makeatletter
\def\blfootnote{\gdef\@thefnmark{}\@footnotetext}
\makeatother

\renewcommand\mathcal{\mathscr}

\usepackage{ulem}
\renewcommand{\emph}{\normalem}

\theoremstyle{plain}
\newtheorem{theorem}{Theorem}[section]
\newtheorem*{theorem*}{Theorem}
\newtheorem{lemma}[theorem]{Lemma}
\newtheorem*{lemma*}{Lemma}
\newtheorem{corollary}[theorem]{Corollary}

\theoremstyle{remark}
\newtheorem{remark}[theorem]{Remark}
\newtheorem*{remark*}{Remark}

\theoremstyle{definition}

\newtheorem*{definition*}{Definition}

\newtheorem{Basic assumptions}[theorem]{Basic assumptions}

\numberwithin{equation}{section}

\newcount\quantno
\everydisplay{\quantno=0}\everycr{\quantno=0}
\newcommand\quant{\advance\quantno by1
                      \ifnum\quantno=1\qquad\else\quad\fi\forall }
\newcommand\itemno[1]{(\romannumeral #1)}

\renewcommand\Im{\operatorname{\mathrm{Im}}}

\newcommand\rest[1]{\kern-.1em
          \lower.5ex\hbox{$\scriptstyle #1$}\kern.05em}

\renewcommand\mod[1]{\vert{#1}\vert}
\newcommand\bigmod[1]{\bigl\vert{#1}\bigr|}

\newcommand\norm[2]{{\Vert{#1}\Vert_{#2}}}

\newcommand\bignorm[2]{\left.{\bigl\Vert{#1}\bigr\Vert_{#2}}\right.}

\newcommand\Bignorm[2]{\left.{\Bigl\Vert{#1}\Bigr\Vert_{#2}}\right.}

\newcommand\wrt{\,\text{\rm d}}

\newcommand\bS{\mathbf{S}}

\newcommand\BC{\mathbb{C}}

\newcommand\BN{\mathbb{N}}
\newcommand\BR{\mathbb{R}}

\newcommand\BT{\mathbb{T}}

\newcommand\BZ{\mathbb{Z}}

\newcommand\cA{\mathcal{A}}

\newcommand\cD{\mathcal{D}}

\newcommand\cF{\mathcal{F}}

\newcommand\cM{\mathcal{M}}

  \newcommand\fT{\mathcal{T}}

\newcommand\ga{\gamma}    \newcommand\Ga{\Gamma}
\newcommand\de{\delta}    \newcommand\De{\Delta}
  \newcommand\vep{\varepsilon}

\newcommand\la{\lambda}   
\newcommand\om{\omega}    \newcommand\Om{\Omega}         
\newcommand\si{\sigma}         

\newcommand\vp{\varphi}

\newcommand\OV{\overline}
\newcommand\funnyk{k\hbox to 0pt{\hss\phantom{g}}}

\newcommand\lu[1]{L^1(#1)}
\newcommand\lp[1]{L^p(#1)}

\newcommand\ly[1]{L^\infty(#1)}
\newcommand\lorentz[3]{L^{#1,#2}(#3)}

\newcommand\Cvp[1]{Cv_p(#1)}
\newcommand\Mp[1]{{\cM}_p(#1)}
\newcommand\Md[1]{{\cM}_2(#1)}

\newcommand\bc{\mathbf{c}}

\newcommand\wt{\widetilde}
\newcommand\whH{\widehat{\phantom{G}}\hbox to 0pt{\hss $H$}}

\newcommand\emspace{\hbox to 6pt{\hss}}
\newcommand\ds{\displaystyle}

\newcommand\rmi{\hbox{\rm (i)}}
\newcommand\rmii{\hbox{\rm (ii)}}

\newcommand\rmiv{\hbox{\rm (iv)}}

\newcommand\One{{\mathbf{1}}}

\newcommand\sft[1]{\wt{#1}}

\newcommand\bSp{\bS_{\de(p)}}

\DeclareSymbolFont{EUEX}{U}{euex}{m}{n}

\DeclareSymbolFont{euexlargesymbols}{U}{euex}{m}{n}
\DeclareMathSymbol{\intop}{\mathop}{euexlargesymbols}{"52}
     \def\int{\intop\nolimits}

\DeclareSymbolFont{euexsymbols}     {U}{euex}{m}{n}
\DeclareMathSymbol{\smallint}{\mathop}{euexsymbols}{"52}

\begin{document}

\title[$L^p$ spherical multipliers on homogeneous trees]
{$L^p$ spherical multipliers \\ on homogeneous trees}



\thanks{Work partially supported by PRIN 2010 ``Real and complex manifolds:
geometry, topology and harmonic analysis".
The first two named authors are members of the Gruppo Nazionale per
l'Analisi Matematica, la Probabilit\`a e le loro Applicazioni
(GNAMPA) of the Istituto Nazionale di Alta Matematica (INdAM).
The research of the third named author was carried over while he was \textit{Assegnista di ricerca}
at the Universit\`a di Milano-Bicocca.  He was supported by PRIN 2010 ``Real and complex manifolds:
geometry, topology and harmonic analysis", by Polish funds for
sciences, National Science Centre (NCN), Poland, Research Project
2014\slash 15\slash D\slash ST1\slash 00405, 
and by Foundation for Polish Science - START scholarship.
}

\author[D. Celotto, S. Meda and B. Wr\'obel]
{Dario Celotto, Stefano Meda and B\l a\.zej Wr\'obel}

\address{Dario Celotto:
Dipartimento di Matematica e Applicazioni
\\ Universit\`a di Milano-Bicocca\\
via R.~Cozzi 53\\ I-20125 Milano\\ Italy
\hfill\break
d.celotto@campus.unimib.it}


\address{Stefano Meda:
Dipartimento di Matematica e Applicazioni
\\ Universit\`a di Milano-Bicocca\\
via R.~Cozzi 53\\ I-20125 Milano\\ Italy
\hfill\break
stefano.meda@unimib.it}



\address{B\l a\.zej Wr\'obel:
Mathematical Institute
\\ 	Universit\"at Bonn\\
Endenicher Allee 60\\ D--53115 Bonn\\ Germany
\newline \&
Instytut Matematyczny, Uniwersytet Wroc\l awski,
pl. Grunwaldzki 2/4, 50-384 Wroc\l aw, Poland
\hfill\break
blazej.wrobel@math.uni.wroc.pl}

\maketitle


\begin{abstract}
We characterise, for each $p$ in $[1,\infty) \setminus \{2\}$, 
the class of $L^p$ spherical multipliers on homogeneous trees in terms of 
$L^p$ Fourier multipliers on the torus.   
\end{abstract}

\setcounter{section}{0}
\section{Introduction} \label{s:Introduction}

A homogeneous tree of degree $q$ is a connected graph $\fT$ with no loops
such that any point $x$ of $\fT$ has exactly~$q+1$ neighbours.   
Henceforth we assume that $q\geq 2$.  
We endow~$\fT$ with the counting measure and the natural distance.

Fix an arbitrary reference point $o$ in $\fT$, denote by $G$ the group of isometries of $\fT$ and 
by $G_o$ the stabiliser of $o$ in $G$.  The group $G_o$ is a maximal compact subgroup of $G$.  
The map $g\mapsto g \cdot o$ identifies $\fT$ with the coset space $G/G_o$. 
Thus, a function $f$ on $\fT$ gives rise to a $G_o$-invariant
function $f'$ on $G$ by the formula $f'(g)=f(g\cdot o)$, and every
$G_o$-invariant function arises in this way.
The distance of $x$ from $o$ will be denoted by $\mod{x}$. 
A function $f$ on $\fT$ is called radial if $f(x)$ depends only on $\mod{x}$,
or equivalently if $f$ is $G_o$-invariant, or if $f'$ is $G_o$--bi-invariant.

It is well known that $G$-invariant linear operators on $\lp{\fT}$ correspond to 
bounded linear operators on $\lp{G/G_o}$ given by convolution on the right with
$G_o$--bi-invariant kernels.  We denote by $\Cvp{\fT}$ the space of radial functions 
on~$\fT$ associated to these $G_o$--bi-invariant kernels.  The norm of an element $k$
in $\Cvp{\fT}$ is then defined as the norm of the corresponding operator
on $\lp{G/G_o}$, equivalently as the norm of the associated $G_o$-invariant operator on $\lp{\fT}$, 
and it is denoted by $\bignorm{k}{\Cvp{\fT}}$.  

We also denote by $\Cvp{\BZ}$ the space of the convolution kernels 
associated to the translation invariant operators on $\lp{\BZ}$.  The norm of a function 
$k$ in $\Cvp{\BZ}$ is the $\lp{\BZ}$ operator norm of the corresponding convolution operator.  
Set $\tau := 2\pi/\log q$, and  denote by $\cF$ the Fourier transformation on $\BZ$, given by
\begin{equation} \label{f: Fourier T}
\cF F(s) 
= \sum_{d\in \BZ} \, F(d)\,q^{-ids}
\quant s\in\BT,  
\end{equation}
where $\BT = \BR/(\tau\BZ)$.  We denote by $\Mp{\BT}$ the space of all (bounded)
functions on $\BT$ of the form $\cF k$, where $k$ is in $\Cvp{\BZ}$.  The norm of 
a function $\cF k$ in $\Mp{\BT}$ is then defined to be the norm of $k$ in $\Cvp{\BZ}$.  

The analogue on trees of a celebrated result of J.L.~Clerc and E.M.~Stein \cite{CSt} states
that if $k$ is in $\Cvp{\fT}$, then its spherical Fourier transform $\wt k$ extends to a
bounded holomorphic function on the strip $\bSp$ (see Section~\ref{s: Background material} for 
the definition of $\bSp$).  This necessary condition was sharpened by   
M.~Cowling, Meda and A.G.~Setti \cite[Theorem~2.1]{CMS2}, who proved that if $k$ is in $\Cvp{\fT}$,
then the boundary values $\wt k_{\de(p)}$ of $\wt k$ on the strip $\bSp$ belong to $\Mp{\BT}$. 
They also proved that this condition implies that convolution with $k$ on the right
is a bounded operator on the space of all \emph{radial} functions in $\lp{\fT}$.  
The work of these authors was inspired by previous work of by R.~Szwarc \cite{Sz} and T.~Pytlik \cite{P}. 
In particular, Pytlik showed that a nonnegative radial function
$k$ is in $\Cvp{\fT}$ if and only if $k$ belongs to the Lorentz space $\lorentz{p}{1}{\fT}$.  
This eventually led Cowling, Meda and Setti to prove a sharp form of the Kunze--Stein phenomenon 
on the full group $G$ \cite[Theorem~1]{CMS2} (see also \cite{N} for a previous less precise version 
of this phenomenon), and A.~Veca to complement this result by proving an endpoint for $p=2$.  
We shall prove the following.  

\begin{theorem*}  
Suppose that $p$ is in $[1,\infty)\setminus \{2\}$, and that $k$ is a radial function on $\fT$.  
Then $k$ is in $\Cvp{\fT}$ if and only if its spherical Fourier transform
$\wt k$ extends to a Weyl-invariant function on $\bSp$ and $\wt k_{\de(p)}$ is in $\Mp{\BT}$.  
\end{theorem*}

\noindent
The proof combines techniques from \cite{CMS2} and a generalisation
of a transference result of A.D.~Ionescu \cite{I1} for rank one noncompact symmetric spaces.  

Our paper is organised as follows.  Section~\ref{s: Background material} 
provides some background and preliminary results. 
Section~\ref{s: A general transference result} contains a general transference result, 
which is of independent interest and may be applied to other situations.
The main result is proved in Section~\ref{s: Spherical multipliers on a tree}.

We shall use the ``variable constant convention'', and denote by $C,$
possibly with sub- or superscripts, a constant that may vary from place to 
place and may depend on any factor quantified (implicitly or explicitly) 
before its occurrence, but not on factors quantified afterwards.

\section{Background material and preliminary results}
\label{s: Background material}



We now summarise the main features of spherical harmonic analysis on~$\fT$. 
Standard references concerning harmonic analysis on trees are the books \cite{FTP,FTN}.  
Our notation is consistent with that of the papers \cite{CMS1,CMS2,CMS3}.
The reader is also referred to the papers
\cite{CS,MS1,MS2,Se1,Se2} for various related aspects of harmonic analysis on homogeneous trees.   
The spherical functions are the radial eigenfunctions of the standard nearest neighbour Laplacian
satisfying the normalisation condition $\phi (o)=1$, and are given by the formula
$$
\phi_z (x) = 
\begin{cases}
\ds\Big(1 + \frac{q-1}{q+1} \mod{ x }\Big) \, q^{-\mod{ x }/2}
&         \hbox{$\quant z\in \tau{\BZ}$} \\
\ds\Big ( 1 + \frac{q-1}{q+1} \mod{ x }\Big) 
                \, q^{- \mod{ x }/2} (-1)^{\mod{ x}}
&         \hbox{$ \quant z\in \tau/2 +\tau{\BZ}$} \\
\ds\bc (z) \, q^{(iz-1/2) \mod{ x}} + \bc (-z)\, q^{(-iz-1/2)\mod{x}}
&        \hbox{$ \quant z\in \BC\setminus (\tau/2) {\BZ}$}, \\
\end{cases}
$$
where $\tau := 2\pi/\log q$ and $\bc$ is the meromorphic function defined by the rule
$$
\bc (z) 
= \frac{q^{1/2}}{q+1} \,
\frac{q^{1/2+iz} - q^{-1/2-iz} }{q^{iz} -  q^{-iz}}
\quant z \in \BC\setminus (\tau/2) {\BZ}.
$$
It is straightforward to check that for each $x$ in $\fT$ the function $z \mapsto \phi_z (x)$ 
is entire and that 
$$
\bigmod{\phi_z (x)} \leq 1
\quant x \in \fT \quant z \in \OV\bS_{1/2}.
$$
For each $p$ in $[1, \infty]$ we write $\de(p)$ for $\bigmod{1/p -1/2}$ and $p'$ for 
the conjugate index $p/(p-1)$.
For any nonnegative real number $t$, we denote by $\bS_t$ and $\OV \bS_t$ 
the strip $\{ z \in {\BC}: \mod{\Im z } < t \}$ and its closure, respectively.
If $f$ is a holomorphic function on $\bS_t$, and $v$ is in $(-t,t)$
then~$f_v$ denotes the function on $\BR$ defined by $f_v(u) = f(u+iv)$.
We also denote by $f_{t}$ and $f_{-t}$ the boundary values of $f$, 
when they exist in the sense of distributions.

The spherical Fourier transform $\sft f$ of a radial function $f$ in 
$\lu{\fT}$ is 
$$
\sft f (z) = \sum_{x\in\fT} f(x) \, \phi_z (x) 
\quant z \in \OV\bS_{1/2}.
$$
Since the map $z \mapsto \phi_z$ is even and $\tau$-periodic in the strip $\bS_{1/2}$,  
so is the function $\sft f$.  
We say that a holomorphic function in a strip $\bSp$ is \emph{Weyl-invariant}  
if it satisfies these conditions in $\bSp$.
We denote the torus $\BR/\tau \BZ$ by ${\BT}$, and usually identify it with $[-\tau/2,\tau/2)$. 
Set $\ds c_{{}_G} = \frac{q\log q}{4\pi (q+1)}$.
It is well known \cite[formula~(3), p.~55]{CMS3} that the inversion formula for 
the spherical Fourier transform may be written as follows:
\begin{equation} \label{f: modified inversion}
f(x)
=2\,c_{{}_G}\, q^{-\mod{x}/2} \, \int_{\BT} \wt f(s)\,\bc (-s)^{-1}\, q^{is\mod{x}}\wrt s.  
\end{equation}  
Recall that the Fourier transformation $\cF$ on $\BZ$ has already been defined (see \eqref{f: Fourier T}).
The corresponding inversion formula is  
$$
F(d) = {1\over \tau} \int_{\BT} \cF F(s) \,q^{ids}\, \wrt s
\quant d \in \BZ.  
$$
Clearly $\cF F$ is $\tau$-periodic on $\BR$. 
A distribution $m$ on $\BT$ is said to be in $\Mp{\BT}$ if convolution 
with $\cF^{-1}m$ defines a bounded operator on $\lp{\BZ}$. 
Note that $\Mp{\BT}$ is contained in $\ly{\BT}$, because 
trivially $\Mp{\BT}$ is contained in $\Md{\BT}$, and $\Md{\BT}$
may be identified with $\ly{\BT}$.  

A geodesic ray $\ga$ in $\fT$ is a one-sided sequence 
$\{\ga_n \colon n\in\BN \}$ of points of $\fT$ such that
$d(\ga_i,\ga_j)=|i-j|$ for all nonnegative integers $i$ and $j$. 
We say that $x$ lies on $\ga$ if $x=\ga_n$ for some $n$ in $\BN$. 
Geodesic rays $\{\ga_n \colon n\in\BN \}$ and $\{\ga'_n \colon n\in\BN \}$ 
are identified if there exist integers $i$ and $j$ such that 
$\ga_n=\ga'_{n+i}$ for all $n$ greater than $j$; this identification is 
an equivalence relation. We denote by $\Om$ the set of the equivalence 
classes, which we call boundary of $\fT$, and by $\Om_x$ the set of all
geodesic rays starting at $x$. Note that for every element $\om$ in $\Om$ 
there exists a unique representative geodesic ray in $\Om_x$: we denote 
this geodesic ray by $[x,\om)$. Given two geodesic rays $\ga^+=[x,\om^+)$ and
$\ga^-=[x,\om^-)$ with intersection $\ga^+\cap\ga^-=\{x\}$ we define the
doubly infinite geodesic $\ga=\{\ga_j\colon j\in\BZ\}$ as follows:
$\ga_j = \ga^+_j$ if $j\geq 0$ and $\ga_j = \ga^-_j$ if $j<0$.
If $\om^+$ and $\om^-$ are two elements of $\Om$ there exists a unique 
(up to renumbering) 
geodesic $\{\ga_j\colon j\in\BZ\}$ such that $\om^+$ and $\om^-$ are the 
equivalence classes of $\{\ga_j\colon j\in\BN\}$ and $\{\ga_{-j}\colon j\in\BN\}$ 
respectively. For brevity, we denote this geodesic by $(\om^+,\om^-)$ disregarding 
the labels.


We fix a reference geodesic $\ga=(\om^-,\om^+)$ such that $o$ lies on $\ga$,
and assume that $\ga$ is indexed so that $\ga_0=o$.
Define the height function $h$ (associated to $\om^+$) by the rule
$$
h_{\om^+}(x)
:= \lim_{i \to\infty}(i-d(x,\ga_i)) \quant x\in\fT.
$$
The level sets of the height function are called {\it horocycles} of $\fT$.


We choose (once and for all) an isometry $\si$ of $\fT$ that maps 
$\ga_i$ in $\ga_{i+1}$ for every $i$.  Then, for $j$ in $\BZ$, $\si^{j}$ is 
an isometry of $\fT$ that maps $\gamma_i$ to $\gamma_{i+j}$. The group $G$ 
admits an Iwasawa-type decomposition $G=NAG_o$, investigated in \cite{FTN,V}. 
Denote by $A$ the subgroup of $G$ generated by the one-step translation 
$\si$ and by $N$ the subgroup of $G$ of all the elements that stabilise $\om^+$ 
and at least an element of $\fT$. It is known that $N$ can be characterised 
as the subgroup of~$G$ consisting in the elements that fix all the horocycles with respect to $\om^+$
\cite[Lemma 3.1]{V}. Furthermore, the orbit of an element $x$ of $\fT$
under the action of $N$ is the horocycle which contains $x$ \cite[Corollary 3.2]{V}.

We endow the totally disconnected group $G$ with the Haar measure
such that the mass of the open subgroup $G_o$ is 1.
Thus,  
\begin{equation*}
\int_G f'(g\cdot o)\wrt g 
= \sum_{x\in\fT} f(x)
\end{equation*}
for all finitely supported functions on $\fT$.
The reader can find much more on the group~$G$ in the book of 
A.~Fig\`a-Talamanca and C.~Nebbia~\cite{FTN}.
It is well known that the group~$N$ is unimodular; we normalise its Haar 
measure $\mu$ by requiring that $\mu(N\cap G_o)=1$, as in 
\cite[Lemma 3.3]{V}. The analogy between $G$ and semisimple Lie groups of 
rank one is apparent in the following theorem \cite[Theorem 3.5]{V}.

\begin{theorem}
Let $G$, $N$, $G_o$ and $\si$ be as above. Then for every $g$ in $G$ there exist 
$n$ in $N$, $j$ in $\BZ$ and $g_o$ in $G_o$ such that $g=n\si^j g_o$. Furthermore, if
$f$ is a continuous compactly supported function on $G$, then
$$
 \int_G\, f(g) \wrt g=\int_N\, \sum_{j\in\BZ}\, q^{-j}\, 
 \int_{G_o}\, f(n \si^j g_o)  \wrt g_o  \wrt\mu(n).
$$
\end{theorem}

\noindent
We remark that, contrary to what happens in the case of noncompact symmetric spaces,
there is a lack of uniqueness in this Iwasawa-type decomposition.
Indeed, if $g = n\si^jg_o = v\si^{\ell}h_o$, then 
$j=\ell$ and there exists $n_o$ in $N\cap G_o$ such that 
$v=\si^jn_o\si^{-j}$ and $h_o=n_o^{-1}g_o$ (see \cite[Remark 3.6]{V}). 

Going back to the tree, a vertex $x$ is of the form $n\si^j\cdot o$, with
$n$ in $N$ and $j$ in $\BZ$. It is straightforward to prove that the height of $x$ 
(with respect to $\om^+$) is simply $j$.
The next lemma establishes a relation between the height of a point and its distance 
from the origin, and may be seen as an analogue of \cite[Lemma 3]{I1}.

\begin{lemma}\label{l: CoorTran}
For every $n$ in $N$ and for every $j$ in $\BZ$ such that $j \leq d(n\cdot o,o)$ 
$$
d(n\si^j\cdot o,o)
= d(n\cdot o,o)-j.
$$
In particular, this formula holds for every $n$ in $N$ and every nonpositive $j$ in $\BZ$.  
\end{lemma}

\begin{proof}
Write $x$ instead of $n\si^j\cdot o$, and denote by $\ga_{\ell}$ the confluence point of $[x,\om^+)$ in 
$\om$, i.e. $[\ga_{\ell},\om_+)=[x,\om_+)\cap\om$ (see also \cite[pag. 6]{CMS2}). 
Note that by definition $\ga_{\ell}$ lies on $[x,\om^+)$, so $\ell\geq j$. We 
observe that such $\ga_{\ell}$ exists, because, by the definition of $N$,
every element of this group fixes a geodesic ray equivalent to $[\ga_j,\om_+)$. 

On a tree the union of two geodesic segments with one extreme in common (but no 
other point) is again a geodesic segment, so
\begin{equation}\label{f: distance}
 d(x,o)=d(x,\ga_{\ell})+d(\ga_{\ell},o).
\end{equation}
Note that $d(\ga_{\ell},o)$ is the absolute value $|\ell|$, as $o$ lies on
the geodesic $\ga$. Moreover, $d(x,\ga_{\ell})$ is always equal to $\ell-j$,
as we already noted that $\ell\geq j$.

Now we consider the cases where $\ell\leq 0$ or $\ell >0$ separately.
If $\ell \leq 0$, then $n$ fixes the origin and \eqref{f: distance} becomes 
$$
d(x,o)
= (\ell-j)-\ell=-j=d(n\cdot o,o)-j.
$$
Otherwise $\ell>0$, and we have $d(n\cdot o,o)=2\ell$, because $n\cdot o$ 
belongs to the same horocycle as $o$. Hence \eqref{f: distance} becomes  
$$
d(x,o)
= (\ell-j)+\ell
= 2\ell-j
= d(n\cdot o,o)-j.  
$$
This concludes the proof of the lemma.
\end{proof}

\noindent
Let $N$ and $A$ be the subgroups of $G$ defined above, and
consider the semi-direct product $NA$, where $A$ acts on $N$ by conjugation.
By \cite[Lemma 3.8]{V} the modular function $\De_{NA}$ of $NA$ is given by 
\begin{equation}\label{f: modular function}
\De_{NA}(n\si^j)=q^{-j} \quant n\in N \quant j\in\BZ.
\end{equation}
By \cite[Theorem 3.5]{V}, we may also identify the convolution
between a $G_o$--right-invariant and $G_o$--bi-invariant 
functions on $G$ with the convolution of the corresponding functions on the group $NA$.
Explicitly, suppose that $f$ is a $G_o$--right
invariant function and that $k$ is a $G_o$--bi-invariant function on~$G$. Then, for $v\in N,$ $j\in \BZ,$ and $g_o\in G_o,$ we have
$$
\begin{aligned}
f*_G k(v\si^jg_0)
& = \int_N\, \sum_{\ell\in\BZ}\, q^{-\ell}\, \int_{G_o}\, f(n \si^\ell h_o)
      \, k(h_o^{-1}\si^{-\ell}n^{-1} v \si^j g_o)  \wrt h_o  \wrt\mu(n) \\
&= \int_N\, \sum_{\ell\in\BZ}\, \De_{NA}(n\si^{\ell})\, f(n \si^\ell)\, k(\si^{-\ell}n^{-1} v \si^j)
      \wrt\mu(n)\\
&= f*_{NA} k (v\si^j),
\end{aligned}
$$
where we have used 
the fact that $G_o$ has total mass $1$. 
By \cite[Theorem 3.5]{V}, the norms of $k$ in $\Cvp{G}$ and in $\Cvp{NA}$ coincide.

For $p$ in $[1,\infty)$, we denote by $Q_p:N\to \BR$ the function defined by 
\begin{equation} \label{f: Qp}
Q_p(n)=q^{-|n\cdot o|/p}.
\end{equation} 

\begin{lemma}\label{l: PoissKer}
Suppose that $p$ is in $[1,2)$.
Then the function $n\mapsto \mod{n\cdot o}^{\ell}\, Q_p(n)$ 
belongs to $\lu{N}$ for each nonnegative integer $\ell$.
\end{lemma}

\begin{proof}
For any nonnegative integer $r$, set $T_r := \{v\in N:  v\cdot o\in S_r(o)\}$. 
By \cite[Lemma 3.11]{V}, $\mu(T_r)$ vanishes if $r$ is odd, 
is equal to $1$ if $r=0$, and is equal to $q^{r/2}$ if $r$ is even and nonzero.  Then,
at least if $\ell \geq 1$,  
$$
\begin{aligned}
\int_N \mod{n\cdot o}^{\ell}\,Q_p(n)  \wrt\mu(n) 
& = \sum_{r\geq 1} \, r^\ell \, q^{-r/p}\mu(T_r)+1\\
& = \sum_{j\geq 1} \, (2j)^\ell \, q^{j(1-2/p)}+1,
\end{aligned}
$$
which is convergent, because $1\leq p<2$, as required.     
\end{proof}

\section{A general transference result}
\label{s: A general transference result}

Denote by $\Ga$ a locally compact group, with left Haar measure $\la$. 
Integration will be with respect to $\la$, unless otherwise specified.  
We denote by $\De_{\Ga}$ the modular function on $\Ga$,
i.e. the Radon--Nykodim derivative of $\la$ with respect to the right Haar measure. 
Given ``nice'' functions $f$ and $\kappa$ on $\Ga$, their convolution on $\Ga$ is defined by 
$$
f*\kappa (x) 
= \int_{\Ga}\, f(y)\, \kappa(y^{-1}x) \wrt\la(y),
$$
Recall the following basic convolution inequality 
$$
\bignorm{f*\kappa}{\lp{\Ga}}
\leq \bignorm{f}{\lp{\Ga}} \bignorm{\De_{\Ga}^{-1/p'}\kappa}{\lu{\Ga}}.  
$$
(see e.g. \cite[Corollary~20.14~\rmii\ and \rmiv]{HR}).  
We denote by $\Cvp{\Ga}$ the space of bounded {\it right} convolutors  of
$\lp{\Ga}$. This space is equipped with the operator norm
$$
\bignorm{\kappa}{\Cvp{\Ga}}
= \sup_{\norm{f}{\lp{\Ga}}=1} \bignorm{f*\kappa}{\lp{\Ga}}.
$$
In this section we assume that the locally compact group $\Ga$ is the semi-direct product of two groups 
$M$ and $H$, where $M$ is normal in $\Ga$ and $H$ acts on $M$ by 
conjugation.  Right Haar measures on $M$ and $H$ will be denoted by $\wrt n$ and $\wrt h$, 
respectively. Then $\wrt g=\wrt n\wrt h$ is a right Haar measure on 
$\Ga$. We denote by $\De_{M}$ and $\De_{H}$ the modular functions of 
$M$ and $H$, so that $\wrt\la(n)=\De_{M}(n) \wrt n$ and 
$\wrt\la(h)=\De_{H}(h) \wrt h$ are left Haar measures on $M$ and $H$,  
respectively.  Note that there is a slight abuse of notation here, for 
$\la$ denotes both a left invariant measure on $H$ and a left Haar measure on $M$. 
 
For $h$ in $H$ and $n$ in $M$, denote by $n^h$ the conjugate $hnh^{-1}$,
and by $\cD(h)^{-1}$ the Radon--Nykodim derivative $\wrt(n^h)/\wrt n$.
It is not hard to check that $\cD$ is an homomorphism of $H$, i.e. 
$\cD(hh')=\cD(h)\, \cD(h')$ for every $h$ and $h'$ in $H$.

\begin{remark}\label{rem: cD with left Haar measures}
Observe that 
$
\cD(h)^{-1}
= \wrt\la(n^h)/\wrt\la(n).  
$
Indeed, note that the conjugation by $h$ commutes with the inversion on $M$, i.e. 
$(n^h)^{-1}=(n^{-1})^h$.  Hence 
$$
\int_{M}\, f(n^h) \wrt\la(n)
= \int_{M}\, f\big(\big((n^{-1})^h\big)^{-1}\big) \wrt\la(n).
$$
Now, the inversion in $M$ transforms the left Haar measure to
the right Haar measure, and conversely.  Thus,  
$$
\begin{aligned}
\int_{M}\, f\big(\big(n^{-1})^h\big)^{-1}\big) \wrt\la(n)
& = \int_{M}\, f\big((v^h)^{-1}\big) \wrt v \\
& = \cD(h)\, \int_{M}\, f(n^{-1}) \wrt n \\
& = \cD(h)\, \int_{M}\, f(v) \wrt\la(v).
\end{aligned}
$$
This fact will be used repeatedly in the proof of Theorem~\ref{t: genTransf}.
\end{remark}

\begin{remark}\label{rem: extension of cD}
Observe that $\cD$ may be extended to a homomorphism on the whole group $\Ga$, by setting
$\cD(nh) := \cD(h)$ for all $n$ in $M$ and $h$ in $H$. 
Recall that $(nh)(n_1h_1) =nn_1^{h}hh_1$.  Thus, 
$$
\begin{aligned}
\cD\big((nh)(n_1h_1)\big)
& = \cD\big(nn_1^{h}hh_1\big)
      = \cD(hh_1)
      = \cD(h)\, \cD(h_1)  \\
& = \cD(nh)\, \cD(n_1h_1). 
\end{aligned}
$$
This observation applies to any homomorphism of $H$.  
\end{remark}

\noindent
It is well known that $\cD(h)\, \De_{M}(n)\, \De_{H}(h) \wrt n\wrt h$ is a left 
Haar measure on~$\Ga$ (see \cite[p. 211]{HR}), and that the following integral formulae hold
$$
\begin{aligned}
\int_{\Ga} f(g) \wrt \la(g) 
& = \int_{M} \int_{H}\, f(nh) \, \cD(h)\, \De_{M}(n)\, \De_{H}(h) \wrt n \wrt h \\
& = \int_{H} \int_{M}\, f(hn) \wrt n \wrt h.  
\end{aligned}
$$ 
The space $\lu{M;\Cvp{H}}$ is the set of all distributions $\kappa$ on $\Ga$  
such that for (almost) every $n$ in $M$ the distribution $\kappa(n \cdot)$ 
induces a bounded convolution operator on $\lp{H}$,
and the function $n \mapsto\bignorm{\kappa(n \cdot)}{\Cvp{H}}$ is in $\lu{M}$.
The space $\lu{M;\Cvp{H}}$ is endowed with the norm 
\begin{equation} \label{f: iterated spaces}
\bignorm{\kappa}{\lu{M;\Cvp{H}}}
:= \int_{M} \, \bignorm{\kappa(n \cdot)}{\Cvp{H}} \wrt\la(n).
\end{equation}


\begin{theorem}\label{t: genTransf}
Suppose that $p$ is in $(1,\infty)$ and that $\De_{M}^{-1/p'} \kappa$ belongs to $\lu{M;\Cvp{H}}$.
Then the operator $f \mapsto f*(\cD^{-1/p} \kappa)$ is bounded on $\lp{\Ga}$, and 
$$
\bignorm{f*(\cD^{-1/p} \kappa)}{\lp{\Ga}}
\leq \bignorm{f}{\lp{\Ga}} \bignorm{ \De_{M}^{-1/p'} \kappa}{\lu{M;\Cvp{H}}}.  
$$
\end{theorem}

\noindent
Notice that 
$$
\bignorm{ \De_{M}^{-1/p'} \kappa}{\lu{M;\Cvp{H}}}
= \int_{M} \bignorm{\kappa(n\cdot)}{\Cvp{H}} \, \De_{M}(n)^{-1/p'}  \wrt \la(n).  
$$

\begin{proof}
Notice that $(nh)^{-1}n_1h_1 = h^{-1}n^{-1}n_1h_1 = (n^{-1}n_1)^{h^{-1}}h^{-1}h_1$.
Thus, 
$$
\begin{aligned}
f*(\cD^{-1/p} \kappa)(n_1h_1)  
 = \int_{M}\int_{H}  f(nh) \, \cD^{-1/p}(h^{-1}h)\, \kappa\big((n^{-1}n_1)^{h^{-1}}h^{-1}h_1\big) 
      \, \cD(h) \wrt\la(n) \wrt\la(h).  
\end{aligned}
$$
We change variables ($(n^{-1}n_1)^{h^{-1}} = m^{-1}$) in the integral over $M$.
Then $m^{-1}=h^{-1}n^{-1}n_1 h$, so that $m=h^{-1}n_1^{-1}nh=(n_1^{-1}n)^{h^{-1}}$, and 
$$
\frac{\wrt\la(m)}{\wrt\la(n)}
= \frac{\wrt\la\big((n_1^{-1}n)^{h^{-1}}\big)}{\wrt\la(n_1^{-1}n)}
  \, \frac{\wrt\la(n_1^{-1}n)}{\wrt\la(n)}
= \cD(h).
$$
The last equality follows from the fact that $\cD$ is a homomorphism 
(whence $\cD(h^{-1})^{-1}=\cD(h)$), and from the left invariance of $\la$.
We conclude that $\!\!\wrt\la(n)=\cD^{-1}(h)\!\! \wrt\la(m)$, whence 
$$
\begin{aligned}
f*(\cD^{-1/p}\, \kappa)(n_1h_1) 
= \int_{M} \!\! \wrt\la(m)\int_{H} f(n_1m^hh)\, \cD^{-1/p}(h^{-1}h_1)\,
          \kappa\big(m^{-1}h^{-1}h_1\big) \wrt\la(h).  
\end{aligned}
$$
We set $U(n_1,m,h) := f(n_1m^hh)$, and 
view the inner integral as the convolution on $H$ between $U(n_1,m, \cdot)$
and $\cD^{-1/p}(\cdot)\, \kappa(m^{-1} \cdot)$, evaluated at the point $h_1$.  
Therefore
$$
\begin{aligned}
& \bignorm{f*(\cD^{-1/p}\, \kappa)}{\lp{\Ga}}  \\
& =    \bigg(\int_{M}\, \int_{H}\, |f*(\cD^{-1/p}\, \kappa)(n_1 h_1)|^p\, \cD(h_1) \wrt\la(h_1) 
         \wrt\la(n_1)\bigg)^{1/p}\\
& =    \Bignorm{\Bignorm{\int_{M} [U(n_1,m,\cdot)*_{H} (\cD^{-1/p}\,\kappa)(m^{-1} \cdot)] (h_1) 
         \, \cD^{1/p}(h_1) \wrt\la(m)}{\lp{H}}}{\lp{M}}
\end{aligned}
$$
where the $\lp{M}$ norm is taken with respect to the left Haar measure of~$M$ and the variable $n_1$. 
Observe that the argument of the integral over~$M$ above may be written as
$$
\int_{H} U(n_1,m,h)\, \cD^{-1/p}(h^{-1}h_1)\, \kappa(m^{-1} h^{-1}h_1)\, \cD^{1/p}(h_1) \wrt h.  
$$ 
Since $\cD$ is an homomorphism, this simplifies to 
$$
\int_{H} U(n_1,m,h)\, \cD^{1/p}(h)\, \kappa(m^{-1} h^{-1}h_1) \wrt h 
=  \big[\big(\cD^{1/p}(\cdot)\, U(n_1,m,\cdot)\big)*_{H}\kappa(m^{-1} \cdot)\big] (h_1).  
$$  
Therefore, by Minkowski's integral inequality,  
\begin{equation} \label{f: tricky writing part 2}
\begin{aligned}
\bignorm{f*(\cD^{-1/p} \kappa)}{\lp{\Ga}}
& \leq \int_{M} \!\!\wrt\la(m)\, \Bignorm{\bignorm{\big(\cD^{1/p}(\cdot)\, U(n_1,m,\cdot)\big)
	*_{H} \kappa(m^{-1} \cdot)}{\lp{H}}}{\lp{M}}. 
\end{aligned}
\end{equation}
By assumption, for every $m$ in $M$ the function $\kappa(m^{-1} \cdot)$ is in $\Cvp{H}$, so that 
$$
\begin{aligned}
& \bignorm{\big(\cD^{1/p}(\cdot)\, U(n_1,m,\cdot)\big) *_{H} \kappa(m^{-1} \cdot)}{\lp{H}}  \\
& \leq\bignorm{\cD^{1/p}(\cdot)\, U(n_1,m,\cdot)}{\lp{H}} \bignorm{\kappa(m^{-1} \cdot)}{\Cvp{H}},
\end{aligned}
$$
and 
$$
\begin{aligned}
& \Bignorm{\bignorm{\big(\cD^{1/p}(\cdot)\, U(n_1,m,\cdot)\big) *_{H} \kappa(m^{-1} \cdot)}{\lp{H}} 
    }{\lp{M}} \\
& \leq \Bignorm{\bignorm{\cD^{1/p}(\cdot)\, U(n_1,m,\cdot)}{\lp{H}}}{\lp{M}}
      \bignorm{\kappa(m^{-1} \cdot)}{\Cvp{H}}.
\end{aligned}
$$
Observe that 
$$
\Bignorm{\bignorm{\cD^{1/p}(\cdot)\, U(n_1,m,\cdot)}{\lp{H}}}{\lp{M}}  
= \Big[\int_{M} \int_{H}\bigmod{f(n_1m^hh)}^p \, \cD(h)\wrt\la(n_1) \wrt\la(h)\Big]^{1/p}.  
$$
We change variables ($n_1m^h = n$) in the inner integral, write
$n_1 m^h=\big(n_1^{h^{-1}}m\big)^{h}$, and observe that 
$$
\begin{aligned}
\frac{\wrt\la(n)}{\wrt\la(n_1)}
&= \frac{\wrt\la\big(\big(n_1^{h^{-1}}m\big)^{h}\big)}{\wrt\la\big(n_1^{h^{-1}}m\big)}\,
   \frac{\wrt\la\big(n_1^{h^{-1}}m\big)}{\wrt\la\big(n_1^{h^{-1}}\big)}\,
   \frac{\wrt\la\big(n_1^{h^{-1}}\big)}{\wrt\la(n_1)}\\
&= \cD(h)^{-1}\, \De_{M}(m)\, \cD(h^{-1})^{-1}\\
&= \De_{M}(m),
\end{aligned}
$$
i.e., $\wrt\la(n_1)=\De^{-1}_{M}(m) \wrt\la(n)$.
Then
$$
\begin{aligned}
\Bignorm{\bignorm{\cD^{1/p}(\cdot)\, U(n_1,m,\cdot)}{\lp{H}}}{\lp{M}}
&= \De_{M}^{-1/p}(m)\bignorm{f}{\lp{\Ga}}.
\end{aligned} 
$$
By combining this and \eqref{f: tricky writing part 2}, we obtain that 
\begin{equation}\label{f: semidirect inequality}
\begin{aligned}
\bignorm{f*\kappa}{\lp{\Ga}}
& \leq \bignorm{f}{\lp{\Ga}}  \int_{M} 
            \bignorm{\kappa(m^{-1} \cdot)}{\Cvp{H}}\, \De_{M}^{-1/p}(m) \wrt\la(m) \\
& = \bignorm{f}{\lp{\Ga}}  \int_{M} 
            \bignorm{\kappa(m \cdot)}{\Cvp{H}}\, \De_{M}^{1/p}(m) \wrt m;
\end{aligned}
\end{equation}
the equality above is a consequence of the change of variables ($m^{-1}\mapsto m$), 
which transforms the left Haar measure into the right Haar measure.
Finally, 
$$
\De_{M}^{1/p}(m) \wrt m
= \De_{M}^{-1/p'}(m)\, \De_{M}(m) \wrt m
= \De_{M}^{-1/p'}(m) \wrt\la(m),
$$ 
which, together with \eqref{f: semidirect inequality}, gives
$$
\bignorm{f*\kappa}{\lp{\Ga}}
\leq \bignorm{f}{\lp{\Ga}}\bignorm{\De_{M}^{-1/p'} \kappa}{\lu{M;\Cvp{H}}},
$$
as required.
\end{proof}

\noindent
We shall apply Theorem~\ref{t: genTransf} when $M$ is unimodular.
For the reader's convenience, we state the 
corresponding results in the following corollary.  

\begin{corollary}\label{cor: transference M unimodular}
Suppose that $p$ is in $(1, \infty)$ and that $M$ is unimodular. Assume that $\kappa$
belongs to $\lu{M;\Cvp{H}}$. Then the operator $f \mapsto f*(\cD^{-1/p}\, \kappa)$ 
is bounded on $\lp{\Ga}$. Furthermore, 
$$
\bignorm{f*(\cD^{-1/p}\, \kappa)}{\lp{\Ga}}
\leq \bignorm{f}{\lp{\Ga}} \bignorm{\kappa}{\lu{M;\Cvp{H}}}.  
$$
\end{corollary}


\section{$L^p$ spherical multipliers on a tree}
\label{s: Spherical multipliers on a tree}

%



In this section we prove our main result.  We first need a lemma on convolutors of $\lp{\BZ}$
whose Fourier transforms extend to holomorphic functions in a strip.  For each positive $\vep$,
we set $\Sigma_\vep:=\{z \in \BC: -\vep < \Im z <0\}$.  

\begin{theorem} \label{t: conv lpBZ}
Suppose that $p$ is in $[1,\infty)$, that $\vp$ is in $\Cvp{\BZ}$, and that
$\cF \vp$ extends to a bounded holomorphic function in the strip 
$\Sigma_\vep$ for some positive $\vep$.  Then 
$$
\bignorm{\vp \, \One_{[J,\infty)}}{\Cvp{\BZ}}
\leq \bignorm{\vp}{\Cvp{\BZ}} + \Big(\frac{1}{q^{\vep}-1} + J\Big) \,\bignorm{\cF\vp}{H^\infty(\Sigma_\vep)}
\quant J \in \BN.  
$$
\end{theorem}

\begin{remark} \label{rem: conv lpBZ}
The conclusion fails for a generic convolutor of $\lp{\BZ}$.  For instance, 
it is well known that the function $\ds \vp(j):= j^{-1} \, \One_{\BZ\setminus\{0\}}(j)$
is in $\Cvp{\BZ}$ for all $p$ in $(1,\infty)$.  However, 
$\vp \One_{[0,\infty)}$ is not in $\Cvp{\BZ}$ for any $p$ in $(1,\infty)$.  
Indeed, if $\vp\One_{[0,\infty)}$ were a 
convolutor of $\lp{\BZ}$, then it would be a finite measure, because 
$\vp\One_{[0,\infty)}$ is nonnegative and $\BZ$ is amenable.  
This contradicts the fact that $\vp\One_{[0,\infty)}$ is nonintegrable on $\BZ$.  
\end{remark}

\begin{proof}
Observe that $\cF \vp$ is $\tau$-periodic in the strip $\Sigma_\vep$.  
A standard argument based on Cauchy's theorem allows us to 
move the path of integration from $[-\tau/2,\tau/2]$ to $[-\tau/2,\tau/2]-i\vep$ 
(the integrals over the vertical sides of the rectangle $[-\tau/2,\tau/2]\times [-\vep,0]$ 
cancel out by periodicity), and obtain that 
$$
\begin{aligned}
\vp(j)
= \frac{1}{\tau} \, \int_{\BT} \cF \vp(s) \, q^{ijs} \wrt s 
= \frac{1}{\tau} \, \int_{\BT} \cF \vp(s-i\vep) \, q^{ij(s-i\vep)} \wrt s.
\end{aligned}
$$
Hence 
$$
\bigmod{\vp(j)}
\leq \bignorm{\cF\vp}{H^\infty(\Sigma_\vep)} q^{j\vep}
\quant j \in \BZ, 
$$
so that $\vp\One_{(-\infty,-1]}$ is integrable on $\BZ$, and 
$$
\bignorm{\vp\One_{(-\infty,-1]}}{\Cvp{\BZ}}
\leq \bignorm{\vp\One_{(-\infty,-1]}}{\lu{\BZ}}
\leq \frac{1}{q^{\vep}-1} \bignorm{\cF\vp}{H^\infty(\Sigma_\vep)}.  
$$ 
Furthermore, trivially 
$$
\bigmod{\vp(j)}
\leq \bignorm{\cF\vp}{\ly{\BT}} 
\quant j \in \BZ, 
$$
whence the function $\vp \One_{[0,J-1]}$ satisfies the estimate
$$
\bignorm{\vp\One_{[0,J-1]}}{\Cvp{\BZ}}
\leq \bignorm{\vp\One_{[0,J-1]}}{\lu{\BZ}}
\leq J \bignorm{\cF\vp}{\ly{\BT}}.
$$ 
As a consequence  
$$
\begin{aligned}
\bignorm{\vp\One_{[J,\infty)}}{\Cvp{\BZ}}
& \leq \bignorm{\vp}{\Cvp{\BZ}} + \bignorm{\vp\One_{(-\infty,-1]}}{\Cvp{\BZ}}
     +   \bignorm{\vp\One_{[0,J-1]}}{\lu{\BZ}} \\
& \leq \bignorm{\vp}{\Cvp{\BZ}} + \frac{1}{q^{\vep}-1} \bignorm{\cF\vp}{H^\infty(\Sigma_\vep)} 
     + J \, \bignorm{\cF\vp}{\ly{\BT}} \\  
& \leq \bignorm{\vp}{\Cvp{\BZ}} + \Big(\frac{1}{q^{\vep}-1} + J\Big) \,  
     \bignorm{\cF\vp}{H^\infty(\Sigma_\vep)},
\end{aligned}
$$
as required.  
\end{proof}

\noindent
Recall that $\de(p) = \bigmod{1/p-1/2}$ and that $\bS_{\de(p)}
= \{z \in \BC: \bigmod{\Im(z)} < \de(p) \}$.  
The main result of this paper is the following.

\begin{theorem}\label{t: Lp1Tree}
Suppose that $p$ is in $[1,\infty)\setminus \{2\}$, and that $k$ is a radial function on $\fT$.  
The following are equivalent:
\begin{enumerate} 
\item[\itemno1]
$k$ is in $\Cvp{\fT}$;
\item[\itemno2]
$\wt k$ is a holomorphic Weyl-invariant function on $\bSp$, and 
$\wt k_{\de(p)}$ is in $\Mp{\BT}$.  
\end{enumerate}
Furthermore, there exists positive constants $c$ and $C$, independent of $k$, such that 
$$
c\, \bignorm{\wt k_{\de(p)}}{\Mp{\BT}} 
\leq \bignorm{k}{\Cvp{\fT}}
\leq C\, \bignorm{\wt k_{\de(p)}}{\Mp{\BT}}.  
$$
\end{theorem}

\begin{proof}
It is known that \rmi\ implies \rmii\ and that the left hand inequality above holds 
(see \cite[Theorem~2.1]{CMS3}).  

Thus, it remains to show that \rmii\ implies \rmi.  
Observe that it suffices to prove the result in the case where $p$ is in $[1,2)$.  Indeed, if 
$p$ is in $(2, \infty)$, and $\wt k_{\de(p)}$ is in $\Mp{\BT}$,
then $\wt k_{\de(p)}$ is also in $\cM_{p'}(\BT)$.  Since $p'$ is in $(1,2)$,
$k$ is in $Cv_{p'}(\fT)$.  A straightforward duality argument then
shows that $k$ is in $\Cvp{\fT}$, as required.  Here we use the fact that $k$ is radial.  

Henceforth we assume that $p$ is in $[1,2)$.  
By the inversion formula \eqref{f: modified inversion}, 
\begin{equation} \label{f: inversion k}
k(x)
=  2 c_{{}_{G}}\, q^{-|x|/2}\, \int_{\BT} \, \wt k(s) \, \bc(-s)^{-1} \, q^{is|x|}  \wrt s 
\quant x\in\fT.
\end{equation} 
The integrand in \eqref{f: inversion k} above is $\tau$-periodic, and holomorphic
in the rectangle $(-\tau/2,\tau/2)\times (-\de(p),\de(p))$.
A standard argument based on Cauchy's theorem allows us to 
move the path of integration from $[-\tau/2,\tau/2]$ to $[-\tau/2,\tau/2]+i\de(p)$ 
(the integrals over the vertical sides of the rectangle $[-\tau/2,\tau/2]\times [0,\de(p)]$ 
cancel out by periodicity), and obtain that 
$$
k(x) 
= 2 c_{{}_{G}}\,q^{-|x|/p} \int_{\BT} \wt k\big(s+i\de(p)\big) \, \bc\big(-s-i\de(p)\big)^{-1}
    \, q^{is|x|}  \wrt s. 
$$
We write $\big(\wt k \check \bc^{-1}\big)_{\de(p)}$ instead of
$\wt k(\cdot +i \de(p))\, \bc(-\cdot -i\de(p))^{-1}$,
and introduce the function~$\vp$ on $\BZ$, defined by 
$$
\vp(\ell)
= 2 c_{{}_{G}}\,\int_{\BT}  \big(\wt k \check \bc^{-1}\big)_{\de(p)} (s)\, q^{is\ell}  \wrt s.  
$$
Then 
\begin{equation} \label{f: formula for k}
k(x) 
=  q^{-|x|/p}\, \vp(|x|).
\end{equation} 
Suppose first that $p=1$.  We must prove that $k$ belongs to $\lu{\fT}$.  Since
$$
\bignorm{k}{\lu{\fT}}
= \sum_{x \in \fT} \, q^{-\mod{x}}\, \bigmod{\vp(|x|)}
= \bigmod{\vp(0)} + \frac{q+1}{q} \, \sum_{d = 1}^\infty \, \bigmod{\vp(d)},
$$
it suffices to prove that $\vp$ is in $\lu{\BZ}$.  Obviously, 
$$
\vp 
= 2c_{{}_G} \tau \, \cF^{-1} \big(\wt k \check \bc^{-1}\big)_{\de(p)} 
= 2c_{{}_G} \tau \, \cF^{-1} \wt k_{\de(p)} *_{{}_\BZ} \cF^{-1}(\check \bc^{-1})_{\de(p)}.
$$
Since $(\check \bc^{-1})_{\de(p)}$ is smooth on $\BT$, its inverse Fourier transform
$\cF^{-1}(\check \bc^{-1})_{\de(p)}$ is in $\lu{\BZ}$, by classical Fourier analysis.  Furthermore
$\cF^{-1} \wt k_{\de(p)}$ is in $\lu{\BZ}$ by assumption.  Therefore~$\vp$ is in $\lu{\BZ}$
and the proof in the case where $p=1$ is complete.  

Now assume that $p$ is in $(1,2)$. 
Denote by $\chi^+$ and $\chi^-$ the functions on $\fT$ defined by 
$$
\chi^+(v\si^j\cdot o) 
= \One_{[0,\infty)} (j) 
\qquad\hbox{and}\qquad 
\chi^-(v\si^j\cdot o) 
= \One_{(-\infty,-1]} (j),
$$
where $v\in N$ and $j\in \BZ.$ Clearly  
$$
k
= k\chi^- +k\chi^+.
$$
It is convenient to view the kernel $k$ as a function on the group $NA$.
In particular, we use formula \eqref{f: formula for k}, 
the definition of $\chi^-$ above, change variables (see Lemma~\ref{l: CoorTran}), recall that
$Q_p(v) = q^{-\mod{v\cdot o}/p}$ (see formula~\eqref{f: Qp}), 
and obtain that 
\begin{equation} \label{f: kchimeno}
\begin{aligned}
\big(k\chi^-\big)(v\si^j\cdot o) 
& = q^{-(|v\cdot o|-j)/p} \, \vp(|v\cdot o|-j) \, \One_{(-\infty,-1]}(j) \\
& = q^{j/p} \, Q_p(v) \, \vp(|v\cdot o|-j) \, \One_{(-\infty,-1]}(j).  
\end{aligned}
\end{equation}  

\emph{Step I: analysis of $k\chi^-$}.
Observe that 
$$
\begin{aligned}
\bigmod{\vp(j)} 
& \leq 2 c_{{}_{G}} \, \int_{\BT}\, |\big(\wt k \check \bc^{-1}\big)_{\de(p)} (s)| \wrt s  \\ 
& \leq 2 c_{{}_{G}} \tau \bignorm{\big(\wt k \check \bc^{-1}\big)_{\de(p)}}{\infty} \\
& \leq 2 c_{{}_{G}} \tau \bignorm{\wt k \check \bc^{-1}}{\ly{\bSp}}
\end{aligned}
$$
for every integer $j$.  As a consequence we obtain the following pointwise bound 
\begin{equation}\label{f: pointwise bound k}
\bigmod{(k\chi^-)(v\si^j\cdot o)} 
\leq 2 c_{{}_{G}} \tau \bignorm{\wt k \check \bc^{-1}}{\ly{\bSp}} \, q^{j/p}\, 
\One_{(-\infty,-1]}(j) \, Q_p(v),  
\end{equation}
which we record for later use.  
Formula \eqref{f: kchimeno} and Corollary~\ref{cor: transference M unimodular} 
(with $\Ga=NA$, $\cD(v \si^j)=q^{-j}$ and $\kappa = \cD^{1/p}k\chi^-$) imply that  
$$
\begin{aligned}
\bignorm{k\chi^-}{\Cvp{NA}}
& \leq \bignorm{\cD^{1/p}k\chi^-}{\lu{N;\Cvp{\BZ}}} \\ 
& =    \int_N  \bignorm{\vp(|v\cdot o|-\cdot)\, \One_{(-\infty,-1]}}{\Cvp{\BZ}} \, Q_p(v) \wrt v.
\end{aligned}
$$
Clearly the norm in $\Cvp{\BZ}$ is translation invariant; it is then straightforward to check that 
$$
\bignorm{\vp(|v\cdot o|-\cdot)\One_{(-\infty,-1]}}{\Cvp{\BZ}} 
= \bignorm{\vp\One_{[\mod{v\cdot o},\infty)}}{\Cvp{\BZ}}.
$$
By Theorem~\ref{t: conv lpBZ} (with $2\de(p)$ in place of $\vep$, and $\mod{v\cdot o}$ in place of $J$) 
$$
\bignorm{\vp\One_{[\mod{v\cdot o},\infty)}}{\Cvp{\BZ}}
\leq  \bignorm{\vp}{\Cvp{\BZ}} + \Big(\frac{1}{q^{2\de(p)}-1}+\mod{v\cdot o}\Big)  
\bignorm{\cF\vp}{H^\infty(\Sigma_{2\de(p)})}.  
$$
By definition of $\vp$ and of the multiplier norm, 
$\bignorm{\vp}{\Cvp{\BZ}} 
= 2 c_{{}_{G}} \tau\bignorm{\big(\wt k\check \bc^{-1}\big)_{\de(p)}}{\Mp{\BT}}$.  
Notice that the function $\check \bc_{\de(p)}^{-1}$ is smooth on $\BR$, and never vanishes.
Therefore there exists a constant $C$ such that 
$$
\bignorm{\big(\wt k\check \bc^{-1}\big)_{\de(p)}}{\Mp{\BT}}
\leq C \bignorm{\wt k_{\de(p)}}{\Mp{\BT}}.  
$$
Furthermore,
$$
\begin{aligned}
(2c_{{}_G}\tau)^{-1} \, \bignorm{\cF\vp}{H^\infty(\Sigma_{2\de(p)})}
& =     \bignorm{(\wt k\check \bc^{-1})_{\de(p)}}{H^\infty(\Sigma_{2\de(p)})}  \\
& =     \max \, \big[\bignorm{(\wt k\check \bc^{-1})_{\de(p)}}{L^\infty(\BT)},
             \bignorm{(\wt k\check \bc^{-1})_{-\de(p)}}{L^\infty(\BT)}\big] \\
& \leq  \bignorm{\check\bc^{-1}}{H^\infty(\bSp)} 
          \, \max \, \big[\bignorm{\wt k_{\de(p)}}{L^\infty(\BT)},
             \bignorm{\wt k_{-\de(p)}}{L^\infty(\BT)}\big] \\
& =     \bignorm{\check\bc^{-1}}{H^\infty(\bSp)} \, \bignorm{\wt k_{\de(p)}}{L^\infty(\BT)} \\
& \leq  \bignorm{\check\bc^{-1}}{H^\infty(\bSp)} \, \bignorm{\wt k_{\de(p)}}{\Mp{\BT}};
\end{aligned}
$$
we have used the Weyl-invariance of $\wt k$ in the last equality above.
By combining the formulae above, we obtain that 
$$
\bignorm{k\chi^-}{\Cvp{NA}}
\leq C \bignorm{\wt k_{\de(p)}}{\Mp{\BT}} \int_N  (1+ \mod{v\cdot o}) \, Q_p(v) \wrt v 
\leq C \bignorm{\wt k_{\de(p)}}{\Mp{\BT}};
$$
the last inequality follows from Lemma~\ref{l: PoissKer}. 

\emph{Step II: analysis of $k\chi^+$}.
Recall that the modular function on $NA$ is $\De_{NA}(v\si^j)=q^{-j}$ (see \eqref{f: modular function}).  
Thus, 
$$
\begin{aligned}
\bignorm{\De_{NA}^{-1/p'} k\chi^+}{\lu{NA}}
& = \sum_{j\in\BZ}\, q^{-j}\, q^{j/p'}\, \int_N\, |k\chi^+(v\si^j\cdot o)|  \wrt\mu(v)\\
& = \sum_{j\geq 0}\, q^{-j/p}\, \int_N\, |k(v\si^j\cdot o)|  \wrt\mu(v).
\end{aligned}
$$
Recall that the Abel transform (see \cite{CMS2}) of $\mod{k}$ is  defined by 
$$
\cA(\mod{k})(j) 
= q^{-j/2}\, \int_N\, \bigmod{k(v\si^j\cdot o)} \wrt\mu(v).
$$  
By \cite[Theorem 2.5]{CMS2}, the Abel transform of $\mod{k}$ 
is an even function on $\BZ$, equivalently 
$$
\int_N\, \bigmod{k(v\si^j\cdot o)} \wrt\mu(v)
= q^{j} \, \int_N\, \bigmod{k(v\si^{-j}\cdot o)} \wrt\mu(v).  
$$
Altogether, we see that 
$$
\bignorm{\De_{NA}^{-1/p'} k\chi^+}{\lu{NA}}
= \sum_{j\geq 0}\, q^{j/p'}\, \int_N\, |k(v\si^{-j}\cdot o)|  \wrt\mu(v).
$$
By the pointwise bound \eqref{f: pointwise bound k} the right hand side is dominated by 
$$
2 c_{{}_{G}} \, \tau \bignorm{\wt k \check \bc^{-1}}{\ly{\bSp}} 
\sum_{j\geq 0}\, q^{j/p'}\, q^{-j/p} \, \int_N\, Q_p(v) \wrt\mu(v).  
$$
Now, the integral over $N$ is convergent, because $p>1$ (see Lemma~\ref{l: PoissKer}), 
and so is the series, because $p<2<p'$.  Therefore 
\begin{equation} \label{f: est kappa1}
\bignorm{\De_{NA}^{-1/p'} k\chi^+}{\lu{NA}}
\leq C  \bignorm{\wt k \check \bc^{-1}}{H^\infty(\bSp)}  
\leq C  \bignorm{\wt k}{H^\infty(\bSp)}:
\end{equation}
the last inequality follows from the fact that $\check\bc^{-1}$ is bounded on $\bSp$.  
Since $\wt k$ is bounded and Weyl-invariant on $\bSp$,  
$$
\bignorm{\wt k}{H^\infty(\bSp)}
\leq \bignorm{\wt k_{\de(p)}}{L^\infty(\BT)}
\leq \bignorm{\wt k_{\de(p)}}{\Mp{\BT}}.
$$

\emph{Step III: conclusion}.  
By combining the estimates proved in \emph{Step~I} and \emph{Step~II}, we see that 
there exists a constant $C$, independent of $k$, such that 
$$
\begin{aligned}
\bignorm{k}{\Cvp{NA}}
& \leq \bignorm{k\chi^+}{\Cvp{NA}} + \bignorm{k\chi^-}{\Cvp{NA}} \\
& \leq C \bignorm{\wt k_{\de(p)}}{\Mp{\BT}}.  
\end{aligned}
$$
Since $k$ is radial on $\fT$, $\bignorm{k}{\Cvp{NA}}=\bignorm{k}{\Cvp{\fT}}$.  
Thus, 
$k$ is in $\Cvp{\fT}$ and the required norm estimate holds.  

This concludes the proof of the theorem.
\end{proof}

\end{document}